\documentclass[11pt]{article}

\usepackage{amsmath,amsthm,amssymb,xypic,stmaryrd, graphicx, mathtools}
\usepackage{hyperref}

\theoremstyle{plain}
    \newtheorem{theorem}{Theorem}[section]
    \newtheorem{lemma}[theorem]{Lemma}
    \newtheorem{corollary}[theorem]{Corollary}
    \newtheorem{proposition}[theorem]{Proposition}

 \theoremstyle{definition}
    \newtheorem{definition}[theorem]{Definition}
    
    \newtheorem{example}[theorem]{Example}

    \newtheorem{remark}[theorem]{Remark}

\theoremstyle{remark}

\numberwithin{equation}{section}

 \DeclareMathOperator{\Tr}{Tr}
 \DeclareMathOperator{\tr}{tr}

\DeclareMathOperator{\AS}{AS}

\DeclareMathOperator{\Id}{Id}
\DeclareMathOperator{\erf}{erf}

\DeclareMathOperator{\ind}{index}

\DeclareMathOperator{\End}{End}
\DeclareMathOperator{\Hom}{Hom}
\DeclareMathOperator{\sgn}{sgn}

\DeclareMathOperator{\Pen}{Pen}

\DeclareMathOperator{\Spin}{Spin}

 \DeclareMathOperator{\Ind}{Ind}

\DeclareMathOperator{\vol}{vol}
\begin{document}


\newcommand{\myemph}{\emph}

\newcommand{\Spinc}{\Spin^c}

    \newcommand{\R}{\mathbb{R}}
    \newcommand{\C}{\mathbb{C}}
    \newcommand{\N}{\mathbb{N}}
    \newcommand{\Z}{\mathbb{Z}}
    \newcommand{\Q}{\mathbb{Q}}
    \newcommand{\bT}{\mathbb{T}}
    \newcommand{\bP}{\mathbb{P}}

\newcommand{\g}{\mathfrak{g}}
\newcommand{\h}{\mathfrak{h}}
\newcommand{\p}{\mathfrak{p}}
\newcommand{\kg}{\mathfrak{g}}
\newcommand{\kt}{\mathfrak{t}}
\newcommand{\ka}{\mathfrak{a}}
\newcommand{\XX}{\mathfrak{X}}
\newcommand{\kh}{\mathfrak{h}}
\newcommand{\kp}{\mathfrak{p}}
\newcommand{\kk}{\mathfrak{k}}

\newcommand{\cA}{\mathcal{A}}
\newcommand{\cE}{\mathcal{E}}
\newcommand{\calL}{\mathcal{L}}
\newcommand{\calH}{\mathcal{H}}
\newcommand{\cO}{\mathcal{O}}
\newcommand{\cB}{\mathcal{B}}
\newcommand{\cK}{\mathcal{K}}
\newcommand{\cP}{\mathcal{P}}
\newcommand{\cN}{\mathcal{N}}
\newcommand{\calD}{\mathcal{D}}
\newcommand{\cC}{\mathcal{C}}
\newcommand{\calS}{\mathcal{S}}
\newcommand{\cM}{\mathcal{M}}
\newcommand{\cU}{\mathcal{U}}
\newcommand{\cS}{\mathcal{S}}

\newcommand{\cCM}{\cC}
\newcommand{\PM}{P}
\newcommand{\DM}{D}
\newcommand{\LM}{L}
\newcommand{\vM}{v}

\newcommand{\sumGam}{\textstyle{\sum_{\Gamma}} }

\newcommand{\sigDg}{\sigma^D_g}

\newcommand{\Bigwedge}{\textstyle{\bigwedge}}

\newcommand{\ii}{\sqrt{-1}}

\newcommand{\Ubar}{\overline{U}}

\newcommand{\beq}[1]{\begin{equation} \label{#1}}
\newcommand{\eeq}{\end{equation}}

\newcommand{\Todo}{\textbf{To do}}

\newcommand{\mattwo}[4]{
\left( \begin{array}{cc}
#1 & #2 \\ #3 & #4
\end{array}
\right)
}

\title{A Lefschetz fixed-point formula for noncompact fixed-point sets}

\author{Peter Hochs\footnote{Radboud University, \texttt{p.hochs@math.ru.nl}}}

\date{\today}

\maketitle

\begin{abstract}
We obtain an equivariant index theorem, or Lefschetz fixed-point formula, for isometries from complete Riemannian manifolds to themselves. The fixed-point set of such an isometry may be noncompact. We build on techniques developed by Roe. Key new ingredients are a localised functional on operators with bounded smooth kernels, and an algebra (reminiscent of Yu's localisation algebra) of ``asymptotically local'' operators, on which this functional has an asymptotic trace property.  As consequences, we show that some earlier indices used are special cases of the one we introduce here, and obtain an obstruction to positive scalar curvature.  
\end{abstract}

\tableofcontents

\section{Introduction}

\subsection*{Background}

A Lefschetz fixed-point theorem, or equivariant index theorem, expresses the graded trace of a geometric operator on cohomology, or on the kernel of an elliptic operator, in topological terms involving the fixed-point set of the map inducing such an operator. Important early results are the ones by Atiyah and Bott \cite{AB66, ABI} and by Atiyah, Singer and Segal \cite{Atiyah68, ASIII}. Since then, more than 100 papers on this subject have appeared, with various generalisations and applications. 

Our focus in this paper is a Lefschetz fixed point formula on noncompact manifolds. Such results have been obtained before, but, to our knowledge, in most of these results it is either assumed that the fixed-point set in question is compact, as for example in \cite{CL08, HW}, or one considers a proper group action with compact quotient, as for example in \cite{HW2, LR03, Wangwang}. The main result in this paper is a Lefschetz fixed-point theorem without such compactness assumptions.

The approach is to build on Roe's index theorem on noncompact maniofolds \cite{Roe88I, Roe88II}. He used functionals associated to certain exhaustions of manifolds by compact sets. Roe proved that these functionals are traces on certain smooth precursors of the Roe algebra, which allowed him to extract a number from a $K$-theoretic index, and obtain an index formula.

A key point in our equivariant setting, is that we need to use \emph{localised} versions of Roe's functionals, to obtain a nontrivial result. (More on this below.) This means that the relevant analogue of Roe's functional is no longer a trace. Our solution is to define a suitable algebra of \emph{asymptotically local} paths of operators, on which this functional has an asymptotic trace property. Constructing the right algebra for this purpose was the main technical challenge in constructing an index and obtaining our main result: the algebra should be large enough so that its $K$-theory contains indices of Dirac operators, and small enough for the asymptotic trace property to hold. 

The algebras of ``uniform operators'' in \cite{Roe88I} later led to the development of the Roe algebra, which has found applications in many places. We found it interesting that our search for an algebra on which our localised functional has an asymptotic trace property naturally led to a smooth analogue of Yu's localisation algebras.
See \cite{WillettYu} for an overview of index theory based on Roe algebras and localisation algebras.

\subsection*{Results}

Let $M$ be a complete Riemannian manifold of bounded geometry, and $\varphi\colon M \to M$ an isometry. Let $D$ be a Dirac-type operator on a vector bundle $S \to M$, and assume that  $D$ commutes with a lift $\Phi$ of $\varphi$ to $S$. Suppose that $D$ is odd with respect to a $\Phi$-invariant grading on $S$. Let $(M_j)_{j=1}^{\infty}$ be an increasing sequence of compact subsets of $M$ whose union is all of $M$. We use a neighbourhood $U$ of the fixed-point set $M^{\varphi}$ of $\varphi$, and a functional $I$ on bounded densities on $M$ associated to the sequence of averages
\[
\frac{1}{\vol(M_j \cap U)} \int_{M_j \cap U} \alpha,
\]
for bounded densities $\alpha$. For an operator $A$ with smooth, bounded Schwartz kernel $\kappa$, we define the trace-like number $\Tr^U_{\Phi}(A)$ as the value of $I$ on the density $m \mapsto \tr( \Phi \kappa(\varphi^{-1}(m), m)) dm$, where $dm$ is the Riemannian density.

A central result in this paper is the construction of an algebra $\cA_{-\infty}^L(S)^{\Phi}$ (see Definition \ref{def AL}) of paths of asymptotically local operators $t \mapsto A(t)$, on which
$\Tr^U_{\Phi}$ is a trace ``in the limit $ t\downarrow 0$'' (see Theorem \ref{thm asympt trace}). The Dirac operator $D$  has an index 
\[
\ind^L(D) \in K_0(\cA_{-\infty}^L(S)^{\Phi}).
\]
The asymptotic trace property of $\Tr^U_{\Phi}$  allows us to apply this functional to $\ind^L(D) $. Our main result is an equivariant index formula of the form
\beq{eq index intro}
\Tr^U_{\Phi}(\ind^L(D) ) = \lim_{j \to \infty} \frac{1}{\vol(M_j \cap U)} \int_{(M_j \cap U)^{\varphi}} \AS_{\Phi}(D),
\eeq
where $\AS_{\Phi}(D)$ is the classical Atiyah--Segal--Singer integrand for $D$, and we assume that the right hand side converges. See Theorem \ref{thm index} for the precise formulation, assumptions are that $\varphi$ preserves an orientation and lies in a compact group of isometries.

Both sides of \eqref{eq index intro} depend on $U$. This is deliberate, and
 it is important to allow $U$ to be different from $M$. Indeed, as $\AS_{\Phi}(D)$ is bounded in our setting, the right hand side of \eqref{eq index intro} for $U = M$ is bounded by a constant times
\[
\frac{\vol(M_j^{\varphi})}{\vol(M_j)}.
\]
In general, we expect this ratio to go to zero as $j \to \infty$ in many cases, because $M^{\varphi}$ has lower dimension than $M$ in general. Taking $U$ to be a tubular neighbourhood of $M^{\varphi}$, for example,  is more likely to lead to nontrivial results.

As consequences of \eqref{eq index intro}, we show that the indices from \cite{HW, Roe88I} are special cases of the index we consider here. In the case where $M^{\varphi}$ and $U$ are compact, we find that
\[
\Tr^U_{\Phi}(\ind^L(D) ) =\frac{1}{\vol( U)} \int_{M^{\varphi}} \AS_{\Phi}(D).
\]
And if $M$ is $\Spin$ and has uniformly positive scalar curvature, then the limit
\[
\lim_{j \to \infty} \frac{1}{\vol(M_j \cap U)} \int_{(M_j \cap U)^{\varphi}} \frac{\hat A(M^{\varphi})}{\det(1-\Phi e^{-R})^{1/2}}
\]
equals zero if it converges. (Here $R$ is the curvature of the connection on the normal bundle to $M^{\varphi}$ induced by the Levi--Civita connection.) The latter result is purely geometric, and its statement does not involve the algebras and functionals developed in this paper.

%
%

%

\subsection*{Outline}

In Section \ref{sec prelim}, we introduce the geometric setting we consider, and the functional $\Tr^U_{\Phi}$ that will be used to define a numerical index of Dirac operators. This functional has an asymptotic trace property on the algebra $\cA_{-\infty}^L(S)^{\Phi}$ introduced in Section \ref{sec asympt loc}. This trace property is Theorem \ref{thm asympt trace}, which is proved in Section \ref{sec trace}. We then use this trace property to obtain a numerical index from a $K$-theoretic index of Dirac operators, and obtain an index theorem, in Section \ref{sec index}.

\subsection*{Acknowledgements}

This research was partially supported by the Australian Research Council, through Discovery Project DP200100729, and by NWO ENW-M grant OCENW.M.21.176. 

\section{Preliminaries}\label{sec prelim}

We start by briefly stating a standard definition of the Dirac operators that we consider. Then we 
introduce a certain type of functional 
 on operators with bounded, smooth kernels (Definition \ref{def TrU}). In Section 
 \ref{sec trace}, we show that this functional has a trace-like property (Theorem \ref{thm asympt trace}) on an algebra introduced in Section \ref{sec asympt loc}. That will allow us to extract a numerical index (Definition \ref{def index}) from a $K$-theoretic index (Definition \ref{def index K}) and obtain an index theorem (Theorem \ref{thm index}).

\subsection{Dirac operators}

Let $M$ be a complete Riemannian manifold. Let $S \to M$ be a Hermitian vector bundle. 

Let $c \colon TM \to \End(S)$ be a vector bundle homomorphism such that for all $v \in TM$, 
\[
c(v)^2 = -\|v\|^2.
\]
Let $\nabla^{TM}$ be the Levi--Civita connection on $TM$. 
Let $\nabla$ be a Hermitian connection on $S$ such that for all vector fields $v$ and $w$ on $M$, and all $s  \in \Gamma^{\infty}(S)$,  
\[
\nabla_v(c(w) s) = c(w) \nabla_v s + c(\nabla^{TM}_v w)s.
\]
Consider the Dirac operator $D\colon \Gamma^{\infty}(S) \to \Gamma^{\infty}(S)$ defined as the composition
\[
D\colon \Gamma^{\infty}(S) \xrightarrow{\nabla} \Gamma^{\infty}(T^*M \otimes S) \xrightarrow{\cong} \Gamma^{\infty}(TM \otimes S)  \xrightarrow{c} \Gamma^{\infty}(S).
\]
The isomorphism in the middle is induced by the isomorphism $T^*M \cong TM$ defined by the Riemannian metric.

We assume that $S$ is $\Z/2\Z$-graded, that $\nabla$ preserves this grading, and that $c(v)$ reverses it, for all $v \in TM$.  
Then $D$ is odd with respect to the grading. 

We assume that $M$ and $S$ have \emph{bounded geometry}, i.e.\ $M$ has positive injectivity radius, and the curvature tensors of $\nabla^{TM}$ and $\nabla$, and all their covariant derivatives, are bounded.

Let $\varphi\colon M \to M$ be an isometry. Let $\Phi\colon S \to S$ be  a smooth map such that for all $m \in M$, the restriction $\Phi_m$ of $\Phi$ to the fibre $S_m$ at $m$ is a linear map from $S_m$  to  $S_{\varphi(m)}$ that preserves the metric and the grading. Then $\varphi$ and $\Phi$ define a map from $\Gamma^{\infty}(S)$ to itself by
\beq{eq Phi action}
(\Phi (s))(m) = \Phi(s(\varphi^{-1}(m)))
\eeq
for all 
 $s \in \Gamma^{\infty}(S)$ and $m \in M$. We suppose that $D$ commutes with this map.

\subsection{Exhaustions}\label{sec exhaust}


We denote the Riemannian distance on $M$ by $d$. 
Let $U$ be an open neighbourhood of the fixed-point set $M^{\varphi}$ such that
\begin{enumerate}
\item $U$ is invariant under $\varphi$; and
\item there is a $\delta>0$ such that $d(\varphi(m), m)\geq\delta$ for all $m \in M \setminus U$.
\end{enumerate}
Our main result, Theorem \ref{thm index} depends on the choice of $U$. This flexibility allows us to choose $U$ such that the main equality \eqref{eq index} yields interesting information. This is crucial, as the most straightforward choice $U = M$ leads to an index theorem that reduces  
 to the equality $0=0$ in many cases, see Remark \ref{rem U=M}. In fact, much of the work done in this paper is to allow the possibility to take $U$ different from $M$, and this is a key difference with \cite{Roe88I}. (Compare for example Theorem \ref{thm asympt trace} in this paper with Proposition 6.7 in \cite{Roe88I}.)

For every $j \in \N$, let $M_j \subset M$ be a compact subset. Suppose that
\begin{enumerate}
\item  $M_{j} \subset M_{j+1}$ for all $j \in \N$; and
\item $\bigcup_{j=1}^{\infty}M_j = M$.
\end{enumerate}
For every $j$, we write
\[
U_j := U \cap M_j.
\]
For $j \in \N$ and $r>0$, define
\[
\Pen_U^-(U_j, r) := \{ m \in U_j; d(m, M\setminus M_j) \geq r \}.
\]
\begin{definition}
The sequence $(M_j)_{j=1}^{\infty}$ is a \emph{$U$-regular exhaustion} of $M$ if for all $r>0$, 
\beq{eq phi reg exh}
\lim_{j \to \infty} \frac{\vol(U_j) - \vol(\Pen_U^-(U_j, r))}{\vol(U_j)}  = 0.
\eeq
\end{definition}

A $U$-regular exhaustion exists for example if volumes of balls in $M$ increase slower than exponentially in their radii, via an argument analogous to the proof of Proposition 6.2 in \cite{Roe88I}.

\begin{example}
Suppose that $M = \R^n$, with the Euclidean metric. Let $k \in \{1, \ldots, n\}$, and define $\varphi\colon M \to M$ by
\[
\varphi(x) = (-x_1, \ldots, -x_k, x_{k+1}, \ldots, x_n)
\]
for $x = (x_1, \ldots, x_n) \in \R^n$. Then $M^{\varphi} = \{0_{\R^k}\} \times \R^{n-k}$. Take $U = (-1,1)^k \times \R^{n-k}$. Then we make take $\delta = 2$ in the second assumption on $U$.

Take $M_j = [-j,j]^n$. Then $U_j = (-1,1)^k \times [-j,j]^{n-k}$. If $r>0$ and $j \geq r+1$, then
\[
\Pen_U^-(U_j, r) = (-1,1)^k \times [-j+r, j-r]^{n-k}.
\]
So for such $j$, 
\[
 \frac{\vol(U_j) - \vol(\Pen_U^-(U_j, r))}{\vol(U_j)} = 
1-\left( \frac{j-r}{j}\right)^{n-k}, 
\]
which goes to $0$ as $j \to \infty$. So  $(M_j)_{j=1}^{\infty}$ is a {$U$-regular exhaustion}. 
\end{example}


\subsection{A functional}\label{sec def trace}

Let $|\Omega_b|(M)$ be the Banach space of bounded, continuous densities on $M$. 
\begin{definition}\label{def associated}
A continuous linear functional $I \colon |\Omega_b|(M) \to \R$ is \emph{associated} with the exhaustion $(U_j)_{j=1}^{\infty}$ of $U$ if for all $\alpha \in |\Omega_b|(M)$,
\[
\liminf_{j \to \infty} \left|  \langle I, \alpha\rangle - \frac{1}{\vol(U_j)} \int_{U_j} \alpha \right| = 0.
\] 
\end{definition}
\begin{lemma}
There exists a functional associated to $(U_j)_{j=1}^{\infty}$. 
\end{lemma}
\begin{proof}
This is analogous to the proof of Proposition 6.4 in \cite{Roe88I}.
\end{proof}

\begin{example}\label{ex associated}
Let $M = \R$, and let $\varphi$ be the identity map. Then we must take $U = M$. Let $U_j = M_j = [-2^{j+1}, 2^{j+1}]$. Let $\chi \in C^{\infty}_c(\R)$ be supported in $[0,1]$, and such that $\int_{0}^1 \chi(x)\, dx = 1$. Define the bounded function $\zeta \in C^{\infty}(\R)$ by
\[
\zeta(x) := \sum_{j=0}^{\infty} (-1)^j \chi(2^{-j}x - 1).
\]
Then the term corresponding to $j$ is supported in $[2^{j}, 2^{j+1}]$, and its integral equals $(-1)^j 2^j$. It follows from the truncated geometric series that
\[
\frac{1}{\vol(U_j)} \int_{U_j} \zeta(x) \, dx = \frac{1 }{3 \cdot 2^{j+2}}  + \frac{(-1)^j}{6}.
\]
This does not converge as $j \to \infty$. But
\[
\liminf_{j \to \infty} \left| \frac{1}{6}  - \frac{1}{\vol(U_j)} \int_{U_j} \zeta(x) \, dx\right| = 0,
\]
and also
\[
\liminf_{j \to \infty} \left| \frac{-1}{6}  - \frac{1}{\vol(U_j)} \int_{U_j} \zeta(x) \, dx\right| = 0.
\]
This suggests that there exist different functionals associated to $(U_j)_{j=1}^{\infty}$, taking the values $\pm 1/6$ on $\zeta dx$. 
\end{example}

We denote the Riemannian density on $M$ by $dm$. 
\begin{definition}\label{def TrU}
Let $I$ be a functional associated to  $(U_j)_{j=1}^{\infty}$. 
Let $A\colon \Gamma^{\infty}_c(S) \to \Gamma^{\infty}(S)$ be a linear operator with a bounded, smooth Schwartz kernel $\kappa$. Define $\alpha \in |\Omega_b|(M)$ by
\[
\alpha(m) := \tr \bigl(\Phi_{\varphi^{-1}(m)} \circ \kappa(\varphi^{-1}(m), m) \bigr) dm.
\]
Then we define
\beq{eq def TrU}
\Tr^U_{\Phi}(A) := \langle I, \alpha \rangle.
\eeq
\end{definition}
\begin{remark}
The number \eqref{eq def TrU} depends on the functional $I$, which is not unique in general. 
\end{remark}

%
%
%
%
%

\section{Asymptotically local operators}\label{sec asympt loc}

We will see in Theorem \ref{thm asympt trace} that the number \eqref{eq def TrU} has an ``asymptotic trace property'' on a suitable algebra of operators. In this section, we construct this algebra. We then define an index of $D$ in the $K$-theory of that algebra.

We do not consider an exhaustion or a functional associated to it until Section \ref{sec trace}.

\subsection{Uniform operators}

We briefly review some definitions from Section 5 of \cite{Roe88I}. 

For $k \in \Z_{\geq 0}$, consider the norm $\|\cdot \|_{W^k}$ on $\Gamma^{\infty}_c(S)$ defined by
\[
\|s\|_{W^k}^2 = \sum_{j=0}^k \|D^j s\|_{L^2(S)}^2,
\]
for $s \in \Gamma_c^{\infty}(S)$. Let $W^k(S)$ be the completion of $\Gamma^{\infty}_c(S)$ in this norm, and let $W^{-k}(S)$ be the continuous dual of $W^k(S)$.

For a subset $X \subset M$, and $s \in W^k(S)$, we define
\[
\|s\|_{W^k, X} := \inf \bigl\{\|s'\|_{W^k}; \text{$s' \in W^k(S)$ is equal to $s$ in a neighbourhood of $X$} \bigr\}.
\]
For a subset $X \subset M$ and $r\geq 0$, we write
\[
\Pen^+(X, r) :=
\overline{\{m \in M; d(m,X) \leq r\}}.
\]
\begin{definition}\label{def uniform}
Let $k \in \Z$. 
An operator $A\colon \Gamma_c^{\infty}(S)\to \Gamma^{\infty}(S)$ is \emph{uniform of order at most $k$} if
\begin{enumerate}
\item for all $l \in \Z$, $A$ extends to a bounded operator from $W^l(S)$ to $W^{l-k}(S)$, also denoted by $A$; and
\item for all $l \in \Z$, there is a function $\mu_l\colon (0, \infty) \to (0, \infty)$ such that
\begin{enumerate}
\item $\lim_{r \to \infty}\mu_l(r) = \infty$; and
\item for all compact subsets $K \subset M$, all $r>0$ and all $s \in W^l(S)$ supported in $K$, 
\[
\| As \|_{W^{l-k}, M \setminus \Pen^+(K, r)} \leq \mu_l(r) \|s\|_{W^k}.
\]
\end{enumerate}
\end{enumerate}
We denote the vector space of all such operators by $\cU_k(S)$. Furthermore, we write
\[
\begin{split}
\cU(S) &:= \bigcup_{k \in \Z} \cU_k(S);\\
\cU_{-\infty}(S) &:= \bigcap_{k \in \Z} \cU_k(S).
\end{split}
\]
\end{definition}

The following fact implies that $\cU(S)$ is an algebra and $\cU_{-\infty}(S)$ is an ideal in $\cU(S)$. 
\begin{lemma}\label{lem Uk alg}
If $A \in \cU_k(S)$ and $B \in \cU_l(S)$, then $AB \in \cU_{k+l}(S)$.
\end{lemma}
This is Proposition 5.2 in \cite{Roe88I}.

\begin{proposition}\label{prop U-infty bdd}
Every operator in $\cU_{-\infty}(S)$ has a bounded, smooth Schwartz kernel.
\end{proposition}
This is the first part of Proposition 5.4 in \cite{Roe88I}. 

%

The algebras $\cU(S)$ and $\cU_{-\infty}(S)$ contain the subalgebras $\cU(S)^{\Phi} \subset \cU(S)$ and $\cU_{-\infty}(S)^{\Phi}\subset \cU_{-\infty}(S)$ of  operators that commute with the map \eqref{eq Phi action} on sections of $S$ defined by $\Phi$ and $\varphi$. Note that $\cU_{-\infty}(S)^{\Phi}$ is an ideal in $\cU(S)^{\Phi}$.

\subsection{Asymptotically local operators}


\begin{definition}\label{def asympt local}
A path of operators $A\colon (0, \infty) \to \cU_{k}(S)$ is \emph{asymptotically local} of order at most $k$ if for all $l \in \Z$,
\begin{itemize}
\item
 there are $C_l, a_l>0$  such that for all $t \in (0,1]$, the norm of $A(t)$ as an operator from $W^l(S)$ to $W^{l-k}(S)$ satisfies
\beq{eq unif norm A}
\|A(t)\| \leq C_l t^{-a_l}, 
\eeq
and 
\item
there is a function $\mu_{l}\colon (0, \infty) \times (0, \infty) \to (0, \infty)$ such that
\begin{enumerate}
\item for all $t>0$, $\lim_{r \to \infty} \mu_l(r, t) = 0$;
\item for all $r>0$ and $a\geq 0$, the function $t\mapsto t^{-a} \mu_l(r,t)$ is bounded and 
%
non-decreasing on $(0,1]$;
\item for all compact subsets $K \subset M$, all $r,t>0$ and all $s \in W^l(S)$ supported in $K$, 
\[
\| A(t) s\|_{W^{l-k}, M \setminus \Pen^+(K, r)} \leq \mu_l(r,t) \|s\|_{W^l}.
\]
\end{enumerate}
\end{itemize}
We denote the vector space of all such operators by $\cU_k^L(S)$. We write
\[
\cU_{-\infty}^L(S) := \bigcap_{k \in \Z}\cU_k^L(S).
\]
\end{definition}

If $A \in \cU_k^L(S)$ and $B \in \cU_l^L(S)$, then we define the path $AB\colon (0, \infty) \to \cU_{k+l}(S)$ by $(AB)(t) = A(t)B(t)$ for $t>0$. Here we use Lemma \ref{lem Uk alg}.
\begin{lemma}\label{lem UkL alg}
For all $A \in \cU_k^L(S)$ and $B \in \cU_l^L(S)$, the path $AB$ lies in $ \cU_{k+l}^L(S)$.
\end{lemma}
\begin{proof}
The condition involving estimates of the form \eqref{eq unif norm A}
 is clearly preserved under composition.

Let $n \in \Z$.
Let $\mu_n^B$ be as in Definition \ref{def asympt local}, applied to $B$ as an operator from $W^n(S)$ to $W^{n-l}(S)$.
 Let $C^B_n, a^B_n>0$ be such that for all $t \in (0,1]$, 
 the norm of $B(t)$ as an operator  from $W^{n}(S)$ to $W^{n-l}(S)$ is at most $C^B_n t^{-a^B_n}$. 
Similarly,  
let $\mu_{n-l}^A$ be as in Definition \ref{def asympt local}, applied to $A$ as an operator from $W^{n-l}(S)$ to $W^{n-l-k}(S)$. 
 Let $C^A_{n-l}, a^A_{n-l}>0$ be such that for all $t \in (0,1]$, 
 the norm of $A(t)$ as an operator  from $W^{n-l}(S)$ to $W^{n-l-k}(S)$ is at most $C^A_{n-l} t^{-a^A_{n-l}}$. 
For $t>0$, define
\[
\begin{split}
F_A(t)&:= \max\{ \|A(t)\|, C^A_{n-l} t^{-a^A_{n-l}}\};\\
F_B(t)&:= \max\{ \|B(t)\|, C^B_{n} t^{-a^B_{n}}\}.
\end{split}
\]
Here $ \|A(t)\|$ is the norm of $A(t)$ as an operator  from $W^{n-l}(S)$ to $W^{n-l-k}(S)$, and $\|B(t)\|$ is the norm of $B(t)$ as an operator  from $W^{n}(S)$ to $W^{n-l}(S)$. (Note that if $t \in (0,1]$, then $F_A(t) =  C^A_{n-l} t^{-a^A_{n-l}}$ and $F_B(t)= C^B_{n} t^{-a^B_{n}}$.)

Define $\mu_{n}^{AB}\colon (0, \infty) \times (0, \infty) \to (0, \infty)$ by
\[
\mu_{n}^{AB}(r,t) := 2 F_B(t)\mu_{n-l}^A(r/2, t) \\
+ \mu_n^B(r/2, t)( F_A(t)+ 2\mu_{n-l}^A(r/2, t)),
\]
for $r,t>0$. It follows from the definitions that this function satisfies the first two conditions in Definition \ref{def asympt local}. In the proof  of Proposition 5.2 in \cite{Roe88I}, it is shown that a version of the function $\mu_{n}^{AB}$ with $F_A(t)$ replaced by $\|A(t)\|$ and $F_B(t)$ replaced by $\|B(t)\|$
 satisfies the third condition on $\mu_n$ in Definition \ref{def asympt local} for the operators $A(t)B(t)\colon W^n(S) \to W^{n-l-k}(S)$. This implies that the function $\mu_{n}^{AB}$ also has this property.
 %
%
%
%
\end{proof}


 If $r>0$ and $m \in M$, then we write $B(m,r) \subset M$ for the open ball with centre $m$ and radius $r$.
\begin{theorem}\label{thm int outside ball}
For all $A \in \cU_{-\infty}^L(S)$, there is a function $v\colon (0, \infty) \times (0, \infty) \to (0, \infty)$ such that 
\begin{enumerate}
\item for all $t>0$, $\lim_{r \to \infty}v(r, t) = 0$;
\item for all $r>0$ and $a\geq 0$, the function $t\mapsto t^{-a} v(r,t)$ is bounded and 
non-decreasing on $(0,1]$;
%
\item
for all $m \in M$, and $r,t>0$, 
if $\kappa_t$ is the Schwartz kernel of $A(t)$, then
\[
\begin{split}
\int_{M \setminus B(m, r)} \|{\kappa_t(m,m')}\|^2\, dm' &\leq v(r,t); \quad \text{and}\\
\int_{M \setminus B(m, r)} \|{\kappa_t(m',m)}\|^2\, dm' &\leq v(r,t).
\end{split}
\]
\end{enumerate}
\end{theorem}
\begin{proof}
In the proof of Proposition 5.4 in \cite{Roe88I}, it is shown that there are $l,l' \in \Z$ and $C,C'>0$ such that for all $t,r>0$,
\[
\begin{split}
\int_{M \setminus B(m, r)} \|{\kappa_t(m,m')}\|^2\, dm' &\leq C\mu_l(r,t)^2; \quad \text{and}\\
\int_{M \setminus B(m, r)} \|{\kappa_t(m',m)}\|^2\, dm' &\leq C'\mu_{l'}(r,t)^2,
\end{split}
\]
with $\mu_l$ and $\mu_{l'}$ as in Definition \ref{def asympt local}. So we can take $v = C\mu_l^2 + C'\mu_{l'}^2$.
\end{proof}

%

We write $\cU_{-\infty}^L(S)^{\Phi}$ for the subalgebra of $\cU_{-\infty}^L(S)$ of paths of operators with values in $\cU_{-\infty}(S)^{\Phi}$.

\subsection{A subalgebra}

Lemma \ref{lem UkL alg} implies that $\cU_{-\infty}^L(S)$ is an algebra. We will use a subalgebra of this algebra.  
By Proposition \ref{prop U-infty bdd}, every operator in $\cU_{-\infty}(S)$ has a smooth, bounded Schwartz kernel. 
\begin{definition}\label{def AL}
The vector space $\cA_{-\infty}^L(S)$ consists of those $A \in \cU_{-\infty}^L(S)$ with the following property. Let $\kappa_t \in \Gamma(S \boxtimes S^*)$ be the smooth kernel of $A(t)$. Then $A \in \cA_{-\infty}^L(S)$ if there are $C,a>0$ such that
 for all  $m \in M$ and $t \in (0,1]$,
\beq{eq int L2 AL}
\begin{split}
\int_{M} \|\kappa_t(m,m')\|^2\, dm'  &\leq  Ct^{-a};\quad \text{and}\\
\int_{M} \|\kappa_t(m',m)\|^2 \, dm'  &\leq  Ct^{-a}.\\
\end{split}
\eeq
\end{definition}
\begin{remark}\label{rem int cond ball}
Because of Theorem \ref{thm int outside ball}, the condition in Definition \ref{def AL} may be replaced by: there exist  $C,a,r>0$ such that for all $m \in M$ and $t \in (0,1]$, 
\beq{eq int L2 eps}
\begin{split}
\int_{B(m,  r)} \|\kappa_t(m,m')\|^2\, dm'  &\leq  Ct^{-a};\quad \text{and}\\
\int_{B(m, r)} \|\kappa_t(m',m')\|^2 \, dm'  &\leq  Ct^{-a}.
\end{split}
\eeq
\end{remark}

\begin{proposition}\label{prop ALS alg}
The space $\cA_{-\infty}^L(S)$ is closed under composition, and hence an algebra.
\end{proposition}
\begin{proof}
Let $A,B \in \cA_{-\infty}^L(S)$. Then $AB \in \cU_{-\infty}^L(S)$ by Lemma \ref{lem UkL alg}. Let $C_A, a_A$ be as in Definition \ref{def AL} for $A$, and let $C_B, a_B$ be as in Definition \ref{def AL} for $B$. 
Let $m \in M$ and $t \in (0,1]$. Let $\kappa_t$ be the Schwartz kernel of $A(t)$, $\lambda_t$ the Schwartz kernel of $B(t)$, and $\mu_t$ the Schwartz kernel of $A(t)B(t)$. Then by the Cauchy--Schwartz inequality, 
\begin{multline*}
\int_{B(m,1)} \|\mu_t(m,m')\|^2\, dm' \\
 \leq  \int_{B(m,1)}
 \left(
 \int_M \|\kappa_t(m,m'')\|^2 \, dm''
 \int_M \|\lambda_t(m'',m')\|^2 \, dm''
 \right)
\, dm'\\
\vol(B(m, 1)) C_A C_B t^{-(a_A+a_B)}.
\end{multline*}
Because $M$ has bounded geometry, $\vol(B(m, 1))$ is bounded in $m$. 
By Remark \ref{rem int cond ball}, it follows that $AB \in \cA_{-\infty}^L(S)$. 
\end{proof}


\begin{lemma}\label{lem AB BA}
Suppose that $A \in \cU_{-\infty}^{L}(S)$, and that  there is a single function $\mu\colon (0, \infty) \times (0, \infty) \to (0, \infty)$ such that for all $l \in \Z$, $\mu_l = \mu$ has the properties in Definition \ref{def asympt local}, and so does the function $(r, t) \mapsto \mu(1,t)^{-1/2} \mu(r,t)$. 
Let $B \in \cU_k(S)$. 
Then $AB$ and $BA$ lie in $\cU_{-\infty}^{L}(S)$. 

If, furthermore, $A$ lies in $\cA_{-\infty}^{L}(S)$ and commutes with $B$, and $k \leq 0$,  then  $AB$ and $BA$ lie in $\cA_{-\infty}^{L}(S)$. 
\end{lemma}
\begin{proof}
Define the path of operators $\tilde B\colon (0, \infty) \to \cU_k(S)$ by $\tilde B(t) = \mu(1,t)^{1/2}B$. Then $\tilde B \in \cU_k^L(S)$. And the path $t \mapsto \mu(1,t)^{-1/2}A(t)$ lies in $\cU_{-\infty}^{L}(S)$ by the assumption on $\mu$. Hence the path $t \mapsto A(t)B = \mu(1,t)^{-1/2}A(t) \tilde B(t)$ lies in $\cU_{-\infty}^{L}(S)$ by Lemma \ref{lem UkL alg}. The same argument applies to $BA$. 

Now suppose that  $A \in \cA_{-\infty}^{L}(S)$. If $k \leq 0$, then $\cU_k(S) \subset \cB(L^2(S))$. Let $\|B\|$ be the operator norm of $B \colon L^2(S) \to L^2(S)$. Let $t>0$. 
Let $\kappa_{A(t)}$ be the Schwartz kernel of $A(t)$ and let  $\kappa_{BA(t)}$ be the Schwartz kernel of $BA(t)$. 
Let $m \in M$. For $v \in S_m$, let $\delta_m v$ be the distributional section of $S$ mapping $s \in \Gamma^{\infty}_c(S^*)$ to $\langle s(m), v\rangle$. Then $\delta_m v \in W^l(S)$ for some $l \in \Z$, so $A(t)\delta_m v \in L^2(S)$. 
Let $\{e_1, \ldots, e_r\}$ be an orthonormal basis of $S_m$. Then, up to a multiplicative constant we absorb in the definition of the norms on $\Hom(S_m, S_{m'})$, for $m' \in M$, we have 
\beq{eq AB BA 1}
\begin{split}
\int_M \|\kappa_{BA(t)}(m', m)\|^2\, dm' &= \sum_{j=1}^r \int_M \|\kappa_{BA(t)}(m',m) e_j\|^2\, dm\\
&= \sum_{j=1}^r \| BA(t) \delta_m e_j  \|_{L^2(S)}^2\\
& \leq \|B\|^2 \sum_{j=1}^r \| A(t) \delta_m e_j  \|_{L^2(S)}^2\\
&= \|B\|^2 \int_M \|\kappa_{A(t)}(m', m)\|^2\, dm'. 
\end{split}
\eeq

%
%

Furthermore, if we use stars to denote $L^2$-adjoints, then, because $B$ and $A(t)$ commute,
\[
\begin{split}
\int_M \|\kappa_{BA(t)}(m, m')\|^2\, dm' &= \int_M \|\kappa_{A(t)B}(m, m')\|^2\, dm' \\
&= \int_M \|\kappa_{(A(t)B)^*}(m', m)\|^2\, dm' \\
&=  \int_M \|\kappa_{B^*A(t)^*}(m', m)\|^2\, dm'.
%
%
\end{split}
\]
As in \eqref{eq AB BA 1}, the latter integral is bounded above by 
\[
\|B^*\|^2 \int_M \|\kappa_{A(t)^*}(m', m)\|^2\, dm' = \|B\|^2 \int_M \|\kappa_{A(t)}(m, m')\|^2\, dm'.
\]
\end{proof}

We write $\cA_{-\infty}^L(S)^{\Phi}$ for the subalgebra of $\cA_{-\infty}^L(S)$ of paths of operators with values in $\cU_{-\infty}(S)^{\Phi}$.

\section{A $K$-theoretic index}

Our next goal is to show that the Dirac operator $D$ has an index in the $K$-theory of the algebra $\cA_{-\infty}^L(S)^{\Phi}$, see Definition \ref{def index K}.

\subsection{Functional calculus}

For $a \in \R$, let  $\cS^a(\R)$ be the space of functions $f \in C^{\infty}(\R)$ such that for all $l = 0, 1, 2, \ldots$, there is a $C_l>0$ such that for all $x \in \R$,
\beq{eq def Sa}
|f^{(l)}(x) | \leq C_l (1+|x|)^{a-l}.
\eeq
For a function $f \in \cS^a(\R)$, the operator $f(D)$ defined by functional calculus lies in $\cU_k(S)$ by Theorem 5.5 in \cite{Roe88II}. We extend this argument to paths of operators in Proposition \ref{prop func calc UkL}. 

\begin{lemma}\label{lem Sa1}
Let $a \in \R$ and $f \in \cS^a(\R)$. 
Let $g \in C^{\infty}_c(\R)$. For $\varepsilon>0$, define the function $f_{\varepsilon} \in C^{\infty}(\R)$ by $f_{\varepsilon}(x) = f(x)g(\varepsilon x)$. 
Then  for all $l\geq 0$ and all $j \geq a+l+2$, there is a $C_{l,j}>0$ such that for all $\varepsilon \in (0,1]$ and $\xi \in \R$,
\[
|\widehat {f_{\varepsilon} }^{(l)}(\xi)| \leq C_{l,j} (1+|\xi|)^{-j}.
\] 
\end{lemma}
\begin{proof}
Let $l \geq 0$. Consider the function $f_{\varepsilon, l}(x) := x^l f_{\varepsilon}(x)$. For all $j \geq 0$,  the Leibniz rule implies that there is a $C>0$ such that for all $\varepsilon \in (0,1]$ and all $x \in \R$, 
\[
|f_{\varepsilon, l}^{(j)}(x)| \leq C(1+ |x|)^{a+l-j}.
\]
If $j \geq a+l+2$, then this becomes
\[
|f_{\varepsilon, l}^{(j)}(x)| \leq C(1+ |x|)^{-2}.
\]
So for all $\xi \in \R$,
\[
|\xi|^j |\widehat{f_{\varepsilon}}^{(l)}(\xi)|  = 
|\widehat{f_{\varepsilon, l}^{(j)}}(\xi)| \leq C \int_{\R}(1+ |x|)^{-2}\, dx.
\]
\end{proof}



\begin{proposition}\label{prop func calc UkL}
Let $k \in \Z$ and $f \in \cS^k(\R)$. We write $\tilde k := \max(-k, 0)$.   Then the path of operators $t \mapsto f(tD)$ lies in $\cU_k^L(S)$. Specifically, there is a single function $\mu\colon (0, \infty) \times (0, \infty) \to (0, \infty)$ such that for all $l \in \Z$, $\mu_l = \mu$ has the properties in Definition \ref{def asympt local},  and for all $b \geq 2\tilde k+2$ there is a constant $C_{b}>0$, independent of $l$,  such that for all $r,t>0$, 
\beq{eq mu fc}
\mu(r,t) \leq C_{b} r^{\tilde k-b+1}t^{b-\tilde k-1}.
\eeq
\end{proposition}
\begin{proof}
For $t>0$, define $f_t \in C^{\infty}(\R)$ by $f_t(x) = f(tx)$. Let $g \in C^{\infty}_c(\R)$ be such that $g(0) = 1$. For $\varepsilon>0$, define $f_{t, \varepsilon} \in C^{\infty}(\R)$ by $f_{t, \varepsilon}(x) = f_t(x)g(\varepsilon x)$.
Let $h \in C^{\infty}(\R)$ be such that $h(\R) \subset [0,1]$, and 
\[
h(\xi) = \left \{ 
\begin{array}{ll}
1 & \text{if $|\xi| \leq 1/2$};\\
0 & \text{if $|\xi| \geq 1$}.
\end{array}
\right.
\]
For $r>0$, define $h_r \in C^{\infty}(\R)$ by $h_r(\xi) = h(\xi/r)$.
By Theorem 5.5 in \cite{Roe88I}, the operator $f(tD)$ lies in $\cU_k(S)$, where the function $\mu_l$ in Definition \ref{def uniform}  may be taken to be
\beq{eq mu l}
\mu_l(r,t) = \mu(r, t) := \frac{1}{2\pi} \sup_{\varepsilon \in (0, t]} \sum_{n=0}^{\tilde k} \int_{\R} \left|  \frac{d^n}{d\xi^n} \left(  \widehat {f_{t, \varepsilon}}(\xi) (1-h_r(\xi))\right) \right| \, d\xi,
\eeq
for all $l \in \Z$.
(The supremum over $\varepsilon$ may be taken over $(0,a]$ for an arbitrarily small $a>0$, because the limit $\varepsilon \downarrow 0$ is taken in the end. The choice $a = t$ is convenient in the current proof.)

By a direct computation,
\[
\widehat {f_{t, \varepsilon}}(\xi) = \frac{1}{t} \widehat {f_{1, \varepsilon/t}}(\xi/t)
\]
for all $\xi \in \R$. So for all $n$ and $\xi$, 
\[
\left|  \frac{d^n}{d\xi^n} \left(  \widehat {f_{t, \varepsilon}}(\xi) (1-h_r(\xi))\right) \right|
\leq \frac{1}{t}\sum_{p=0}^n {n \choose p} 
\left| \frac{d^p}{d\xi^p}  \widehat{f_{1, \varepsilon/t}}(\xi/t) \right|
 \cdot \left| \frac{d^{n-p}}{d\xi^{n-p}}   (1-h_r(\xi))\right|.
\]
Now the factor $\left| \frac{d^p}{d\xi^p}  \widehat{f_{1, \varepsilon/t}}(\xi/t) \right|$ is bounded by a constant times
\[
\frac{1}{t^p}
\left|   \widehat{f_{1, \varepsilon/t}}^{(p)}(\xi/t) \right|.
\]
And the factor $\left| \frac{d^{n-p}}{d\xi^{n-p}}   (1-h_r(\xi))\right|$ equals $|1-h(\xi/r)|$ if $p=n$, and is bounded by a constant times $r^{-(n-p)} h^{(n-p)}(\xi/r)$ if $p<n$. In all cases, it is bounded by a constant times $r^{-(n-p)}$ times a function that vanishes on $[-r/2, r/2]$.

We find that if $t\in (0,1]$ and $r \geq 1$,  \eqref{eq mu l} is smaller than or equal to a constant times
\begin{multline}\label{eq Sa 3}
t^{-\tilde k-1}r^{\tilde k}
\sup_{\varepsilon\in (0, t]} \sum_{n=0}^{\tilde k} \sum_{p=0}^n \int_{\R \setminus [-r/2, r/2]}|\widehat {f_{1, \varepsilon/t}}^{(p)}(\xi/t)|\, d\xi\\
= 
t^{-\tilde k-1}r^{\tilde k}
\sup_{\varepsilon\in (0, 1]} \sum_{n=0}^{\tilde k} \sum_{p=0}^n \int_{\R \setminus [-r/2, r/2]}|\widehat {f_{1, \varepsilon}}^{(p)}(\xi/t)|\, d\xi
\end{multline}
Lemma \ref{lem Sa1} implies that for all $p\geq 0$ and all $b \geq k+p+2$, there is a $C_{p,b}>0$  such that for all $\xi \in \R$ and all $\varepsilon \in (0,1]$,
\[
|\widehat {f_{1, \varepsilon}} ^{(p)}(\xi)| \leq C_{p,b} (1+|\xi|)^{-b}.
\] 
So if $t \in (0,1]$ and $r\geq 1$, and $b\geq 2\tilde k+2$,  then  \eqref{eq Sa 3} is smaller than or equal to a constant times
\[
t^{-\tilde k-1}r^{\tilde k}
 \sum_{n=0}^{\tilde k} \sum_{p=0}^n \int_{\R \setminus [-r/2, r/2]}
(1+|\xi/t|)^{-b}
 \, d\xi.
\] 
The latter expression is smaller than or equal to a constant times
\[
t^{-\tilde k}r^{\tilde k} (1+(r/2t))^{-b+1} < 2^{1-b} t^{-\tilde k+b-1} r^{\tilde k-b+1}.
\]
\end{proof}
%

\subsection{Heat operators}

For $t>0$, consider the heat operator $e^{-tD^2}$ associated to $D$.  
\begin{proposition}\label{prop pos j}
For all $j \geq 0$,  the path of operators $t\mapsto  D^j e^{-t^2D^2}$ lies in $\cU_{-\infty}^L(S)^{\Phi}$.
\end{proposition}
\begin{proof}
For all $t>0$ and $j,l \geq 0$, the function $f_{j+l,t}$ on $(0, \infty)$ defined by
 $f_{j+l,t}(x) = x^{j+l} e^{-t^2 x^2}$ has maximum value
 \[
 \|f_{j+l,t}\|_{\infty} = \left( \frac{j+l}{2t^2}\right)^{(j+l)/2} e^{-(j+l)/2}.
 \]
 So for all $k \in \Z$, the right hand side is an upper bound for the norm of
$D^j e^{-t^2D^2}$ as an operator from $W^k(S)$ to $W^{k+l}(S)$. It follows that $t\mapsto   D^j e^{-t^2D^2}$ has the first property in Definition \ref{def asympt local}, for all $k \in \Z$. 

Let $\tilde \mu$ be as in Proposition \ref{prop func calc UkL}, with $f(x) =  x^j e^{-x^2}$. 
This function has the properties of the functions $\mu_l$ in Definition \ref{def asympt local} for 
%
%
 the path of operators $t\mapsto   t^j D^j  e^{-t^2D^2}$. It follows that the function $\mu(r,t) = t^{-j} \tilde \mu(r,t)$ has these properties for the path of operators $t\mapsto    D^j  e^{-t^2D^2}$.
%
\end{proof}

\begin{proposition}\label{prop pos j A}
For all $j \geq 0$,  the path of operators $t\mapsto  D^j e^{-t^2D^2}$ lies in $\cA_{-\infty}^L(S)^{\Phi}$.
\end{proposition}
\begin{proof}
For $t>0$, let $\kappa_t$ be the Schwartz kernel of $D^j e^{-t^2D^2}$. 
By Proposition 4.2 in \cite{CGRS14}, the fact that $M$ has bounded geometry implies that there are $C, a>0$ such that for all $m,m' \in M$ and all $ t\in (0,1]$,
\[
\|\kappa_{t}(m, m')\| \leq C t^{-(\dim(M)+j)} e^{-a d(m,m')^2/t^2} \leq C t^{-(\dim(M)+j)} .
\]
So for all $m \in M$ and $\varepsilon>0$, the left hand sides of  \eqref{eq int L2 eps} are less than or equal to
\[
C^2 \vol(B(m, r))  t^{-2(\dim(M)+j)}.
\]
Because $M$ has bounded geometry, the volume of $B(m,r)$ is bounded in $m$.
By Remark \ref{rem int cond ball} and Proposition \ref{prop pos j}, this implies the claim.
%
%
%
\end{proof}

\subsection{The index}

\begin{definition}\label{def ALSD}
Let $\cA_{-\infty}^L(S; D) \subset \cA_{-\infty}^L(S)^{\Phi}$ be the linear subspace of $A \in \cA_{-\infty}^L(S)$ for which 
\begin{enumerate}
\item
there are functions $f\colon (0, \infty) \to (0, \infty)$ and $g \in \cS(\R)$ such that for all $t>0$, 
\[
A(t) = f(t) g(t D).
\]
\item  for all $j = 0, 1, 2, \ldots$, there are  $C_j,a_j>0$ such that
 for all  $m \in M$ and $t \in (0,1]$,
\[
\begin{split}
\int_{M} \|\kappa_{j,t}(m,m')\|^2\, dm'  &\leq  C_jt^{-a_j};\quad \text{and}\\
\int_{M} \|\kappa_{j, t}(m',m)\|^2 \, dm'  &\leq  C_jt^{-a_j},
\end{split}
\]
where $\kappa_{j, t}$ is the Schwartz kernel of $D^j A(t)$. 
\end{enumerate}
Let $\cU^L(S; D)$ be the vector space of all paths $A\colon (0, \infty) \to \cU(S)$ such that
\[
\begin{split}
A \cA_{-\infty}^L(S; D) &\subset \cA_{-\infty}^L(S; D); \quad \text{and}\\
  \cA_{-\infty}^L(S; D) A&\subset \cA_{-\infty}^L(S; D).
  \end{split}
\] 
\end{definition}

\begin{lemma}\label{lem ideal AD}
The space $\cU^L(S; D)$ is an algebra, and $\cA_{-\infty}^L(S; D)$ is a two-sided ideal in $\cU^L(S; D)$.
\end{lemma}
\begin{proof}
The only nontrivial point to check is that $\cA_{-\infty}^L(S; D)$ is closed under composition. The first point in Definition \ref{def ALSD} is clearly preserved under composition. One can show that the second point in Definition \ref{def ALSD} is  preserved under composition in a similar way to the proof of Proposition \ref{prop ALS alg}.
\end{proof}

\begin{lemma}\label{lem D Q}
The algebra  $\cU^L(S; D)$ contains the constant path $D$ and the path $Q(t) = \frac{1-e^{-t^2D^2}}{D}$. 
\end{lemma}
\begin{proof}
It follows from Definition \ref{def ALSD} that  $D \in \cU^L(S; D)$. To see that $Q \in \cU^L(S; D)$, consider the function $h(x) = \frac{1-e^{-x^2}}{x}$. Let $A \in \cA_{-\infty}^L(S; D)$, and let $f\colon (0, \infty) \to (0, \infty)$ and $g \in \cS(\R)$ be such that for all $t>0$, 
we have $
A(t) = f(t) g(tD)$.
Then
\[
A(t)Q(t) = f(t) t (hg)(tD).
\]
The Fourier transform of $h$ is
\[
\hat h(\xi) = -i \pi \left( \erf(-\xi/2)+\sgn(\xi) \right).
\]
This function decays faster than any rational function in $\xi$. So $h$ is smooth, and all its derivatives are (square integrable, and hence) bounded.
So $hg \in \cS(\R)$, and therefore $AQ$ is of the desired form.

 It follows from Proposition \ref{prop func calc UkL} that $AQ \in \cU_{-\infty}^L(S)$. The second point in Definition \ref{def ALSD} follows because $Q(t)$ is bounded and commutes with $A(t)$, analogously to the proof of Lemma \ref{lem AB BA}.
\end{proof}

\begin{remark}
Lemma \ref{lem D Q} is the reason why we use the subalgebra $\cA_{-\infty}^L(S; D) \subset \cA_{-\infty}^L(S)^{\Phi}$. It is less obvious to us if $Q$ is a multiplier of $\cA_{-\infty}^L(S)^{\Phi}$. 
\end{remark}

\begin{proposition}\label{prop D Fred}
The image of $D$ in $\cU^L(S; D)/\cA_{-\infty}^L(S; D)$ is invertible.
\end{proposition}
\begin{proof}
Let $Q \in \cU^L(S; D)$ be as in Lemma \ref{lem D Q}. Then for all $t>0$, 
\[
1-DQ(t) = e^{-t^2D^2}.
\]
So $1-DQ \in \cA_{-\infty}^L(S; D)$ by Proposition \ref{prop pos j A}. 
\end{proof}

By Proposition \ref{prop D Fred}, the restriction of $D$ to even-graded sections defines a class
\beq{eq D K1}
[D^+] \in K_1(\cU^L(S; D)/\cA_{-\infty}^L(S; D)).
\eeq
\begin{definition}\label{def index K}
The class
\beq{eq indDL}
\ind_D^L(D) \in K_0(\cA_{-\infty}^L(S; D))
\eeq
is the image of \eqref{eq D K1} under the boundary map
\beq{eq bdry K}
\partial\colon K_1(\cU^L(S; D)/\cA_{-\infty}^L(S; D)) \to K_0(\cA_{-\infty}^L(S; D)).
\eeq
The class
\beq{eq indD}
\ind^L(D) \in K_0(\cA_{-\infty}^L(S)^{\Phi})
\eeq
is the image of \eqref{eq indDL} under the  map 
\[
\iota_*\colon K_0(\cA_{-\infty}^L(S; D)) \to K_0(\cA_{-\infty}^L(S)^{\Phi})
\]
induced by the inclusion map $\iota\colon \cA_{-\infty}^L(S; D) \to \cA_{-\infty}^L(S)^{\Phi}$.
\end{definition}
\begin{remark}
As in \cite{Roe88I}, we use \emph{algebraic} $K$-theory to define the indices \eqref{eq indDL} and \eqref{eq indD}, because we do not consider topologies on the algebras used.
\end{remark}
\begin{remark}
The index \eqref{eq indDL} lies in a $K$-theory group depending on $D$, and is therefore less suitable for constructing canonical invariants of $M$. The  index \eqref{eq indD} lies in a $K$-theory group independent of $D$; we expect that this can even be mapped to a $K$-theory group independent of $S$ analogously to Section 7 in \cite{Roe88I}. However, it is a priori possible that the index \eqref{eq indDL} is nonzero in cases when \eqref{eq indD} is zero, so \eqref{eq indDL}  could potentially be more refined, for example as an obstruction to positive scalar curvature.
\end{remark}

\begin{proposition}\label{prop Dinv ULSD}
If $D$ is invertible, then $\cU^L(S; D)$ contains the constant path $D^{-1}$.
\end{proposition}
\begin{proof}
Fix $c>0$ such that $[-c,c]$ is disjoint from the spectrum of  $D$. 
Let $f \in C^{\infty}(\R)$ be such that
\[
f(x) = \left \{ 
\begin{array}{ll}
0 & \text{if $|x| \leq c/2$};\\
1 & \text{if $|x| \geq c$}.
\end{array}
\right.
\]
Then $x \mapsto f(x)/x$ lies in $\cS^{-1}(\R)$, so $D^{-1} = f(D)/D \in \cU_{-1}(S)$ by Theorem 5.5 in \cite{Roe88I}. By Proposition \ref{prop func calc UkL}, an element $A \in \cA_{-\infty}^L(S; D)$ satisfies the assumptions in Lemma \ref{lem AB BA}. So $D^{-1}A = AD^{-1}$ lies in $\cA_{-\infty}^L(S)$. And this path of operators also satisfies the conditions in Definition \ref{def ALSD}, where for the second condition we use $L^2$-boundedness of $D^{-1}$, analogously to the proof of Lemma \ref{lem AB BA}. So $D^{-1}A = AD^{-1}$ lies in $\cA_{-\infty}^L(S; D)$.
\end{proof}

\begin{proposition}\label{prop inv index zero}
If $D$ is invertible, then the indices \eqref{eq indDL} and \eqref{eq indD} equal zero.
\end{proposition}
\begin{proof}
If $D$ is invertible, then $D^+$ defines a class in $K_1(\cU^L(S; D))$ by Proposition \ref{prop Dinv ULSD}. The image of this class in $ K_1(\cU^L(S; D)/\cA_{-\infty}^L(S; D))$ is \eqref{eq D K1}. Then the image of the latter class under \eqref{eq bdry K} is now zero, because the composition of two consecutive maps in the $K$-theory exact sequence is zero.
\end{proof}

\section{An asymptotic trace property}\label{sec trace}


We return to the situation of Subsections \ref{sec exhaust} and \ref{sec def trace}. Let $U$ be as in Subsection \ref{sec exhaust}, fix a $U$-regular exhaustion $(M_j)_{j=1}^{\infty}$ of $M$, and let $I$ be a functional on $|\Omega_b|(M)$ associated to $(U_j)_{j=1}^{\infty}$. By Proposition \ref{prop U-infty bdd}, every path of operators $A \in \cA_{-\infty}^L(S)$ defines a function
\[
t \mapsto \Tr^U_{\Phi}(A(t))
\]
as in \eqref{eq def TrU}. 

The main result of this section is the following ``asymptotic trace property'' of this construction. This will allow us to extract a numerical index from the $K$-theoretic index of Definition \ref{def index K}. 
\begin{theorem}\label{thm asympt trace}
For all $A \in \cA_{-\infty}^L(S)^{\Phi}$ and $B \in \cA_{-\infty}^L(S)$, 
\[
\lim_{t \downarrow 0} \Tr_{\Phi}^U (A(t)B(t) - B(t) A(t)) = 0.
\]
\end{theorem}

\subsection{Estimates of partial integrals}\label{sec est int}

Let $A,B \in \cA_{-\infty}^L(S)$. 
For $t>0$, let $\kappa_t$ be the Schwartz kernel of $A(t)$, and let $\lambda_t$ be the Schwartz kernel of $B(t)$.

Let $\varepsilon>0$. Let $v_A$ and $v_B$ be the functions in Theorem \ref{thm int outside ball}, applied to $A$ and $B$, respectively. 
Let $r>0$ be such that $\max(v_A(r,1), v_B(r,1))<\varepsilon$. Because of the third and fourth points in Theorem \ref{thm int outside ball}, we then have for all $m \in M$ and $t \in (0,1]$, 
\beq{eq est r kappa lambda}
\begin{split}
\int_{M \setminus B(m, r)} \|{\kappa_t(m,m')}\|^2\, dm' &< \varepsilon; 
\quad \text{and} \\
\int_{M \setminus B(m, r)} \|{\lambda_t(m',m)}\|^2\, dm' &< \varepsilon.
\end{split}
\eeq

For $j \in \N$, 
consider the following subsets of $U_j \times M$:
\beq{eq Vj}
\begin{split}
V_j &:= U_j \times U_j;\\
W_j &:= \{(m,m') \in U_j \times (M \setminus M_j); \text{$d(m,m')<r$ or $d(\varphi^{-1}(m),m')<r$} \};\\
X_j &:= \{(m,m') \in U_j \times (M \setminus M_j); \text{$d(m,m')\geq r$ and $d(\varphi^{-1}(m),m')\geq r$} \};\\
Y_j &:= U_j \times (M_j \setminus U_j).
\end{split}
\eeq
These sets are disjoint, and their union is $U_j \times M$.

\begin{lemma}\label{lem est Vj}
If $A \in \cA_{-\infty}^L(S)^{\Phi}$, then
for all $t>0$ and $j \in \N$, 
\[
\int_{V_j} \tr\left( \Phi \kappa_t(\varphi^{-1}(m), m')\lambda_t(m', m) - \Phi \lambda_t(\varphi^{-1}(m), m') \kappa_t(m', m) \right)\, dm' \, dm = 0.
\]
\end{lemma}
\begin{proof}
Let $t>0$ and $j \in \N$.
Because $A(t)$ commutes with $\Phi$, we have for all $m,m' \in M$,
\[
\Phi \kappa_t(\varphi^{-1}(m), m') =\kappa_t(m, \varphi(m'))  \Phi. 
\]
Hence by the trace property of the fibre-wise trace,  
\begin{multline} \label{eq est Vj 1}
\int_{V_j} \tr\left( \Phi \kappa_t(\varphi^{-1}(m), m')\lambda_t(m', m)  \right)\, dm' \, dm \\
=
\int_{U_j} \int_{U_j} \tr\left(\Phi  \lambda_t(m', m) \kappa_t(m, \varphi(m'))   \right)\, dm \, dm'.
\end{multline}
Here we also used the Fubini--Tonelli theorem to interchange the integrals; we have absolute convergence by compactness of $\overline{U_j}$ and boundedness of $\kappa_t$ and $\lambda_t$ (see Proposition \ref{prop U-infty bdd}). Substituting $m'' = \varphi(m')$ and using $\varphi$-invariance of $dm$, we see that the right hand side of \eqref{eq est Vj 1} equals
\[
\int_{V_j} \tr\left(\Phi  \lambda_t(\varphi^{-1}(m''), m) \kappa_t(m, m'')   \right)\, dm \, dm''.
\]
\end{proof}
\begin{remark}
Lemma \ref{lem est Vj} is the one place in the proof of Theorem \ref{thm asympt trace} where we use $\Phi$-equivariance of $A$. 
\end{remark}

\begin{lemma}\label{lem est Wj}
For all $t>0$, there is an $N>0$ such that for all $j \geq N$, 
\beq{eq est Wj}
\left| \frac{1}{\vol(U_j)}\int_{W_j} \tr\left( \Phi \kappa_t(\varphi^{-1}(m), m')\lambda_t(m', m) \right)\, dm' \, dm
\right| < \varepsilon.
\eeq
\end{lemma}
\begin{proof}
We claim that for all $j$, 
\beq{eq Wj Pen}
W_j \subset \bigl( (U_j \setminus \Pen_U^-(U_j, r) \bigr) \times M.
\eeq
Indeed, suppose that $(m,m') \in W_j$. If $d(m,m')<r$, then $d(m, M \setminus M_j)<r$, so $m \not\in \Pen_U^-(U_j, r)$. And if $d(\varphi^{-1}(m),m') = d(m,\varphi(m'))<r$, then $m \not\in \Pen_U^-(U_j, r)$ because $M_j$ is $\varphi$-invariant. So \eqref{eq Wj Pen} follows.

Because $M$ has bounded geometry, there is a $V>0$ such that for all $m \in M$, the volume of $B(m, r)$ is at most $V$.
It follows from \eqref{eq Wj Pen} and the fact that $\Phi$ preserves the metric that the left hand side of \eqref{eq est Wj} is smaller than or equal to
\begin{multline*}
\frac{1}{\vol(U_j)}\int_{U_j \setminus \Pen_U^-(U_j, r)} \int_{B(m, r) \cup B(\varphi^{-1}(m), r)}
\|  \kappa_t(\varphi^{-1}(m), m') \| \| \lambda_t(m', m) \| \, dm' \, dm\\
\leq 2V \frac{\vol(U_j) - \vol(\Pen_U^-(U_j, r))}{\vol(U_j)} \|\kappa_t\|_{\infty} \|\lambda_t\|_{\infty}. 
\end{multline*}
Here we used boundedness of $\kappa_t$ and $\lambda_t$, Proposition \ref{prop U-infty bdd}. 
So the claim follows from \eqref{eq phi reg exh}.
\end{proof}

\begin{lemma}\label{lem est Xj}
For all $t \in (0,1]$ and all $j \in \N$, 
\beq{eq est Xj}
\left| \frac{1}{\vol(U_j)}\int_{X_j} \tr\left( \Phi \kappa_t(\varphi^{-1}(m), m')\lambda_t(m', m) \right)\, dm' \, dm
\right| < \varepsilon.
\eeq
\end{lemma}
\begin{proof}
Let $t \in (0,1]$ and all $j \in \N$. Then the left hand side of \eqref{eq est Xj} is smaller than or equal to
\begin{multline}
 \frac{1}{\vol(U_j)}\int_{U_j} \int_{M \setminus (B(m, r) \cup B(\varphi^{-1}(m), r) )} \| \kappa_t(\varphi^{-1}(m), m')\lambda_t(m', m) \| \, dm' \, dm\\
 \leq
  \frac{1}{\vol(U_j)}\int_{U_j} 
  \left( \int_{M \setminus B(\varphi^{-1}(m), r) } \| \kappa_t(\varphi^{-1}(m), m') \|^2 \, dm'  \right)^{1/2}
  \\
    \left( \int_{M \setminus B(m, r)} \| \lambda_t(m', m) \| \, dm'  \right)^{1/2}
  \, dm.
\end{multline}
By 
 \eqref{eq est r kappa lambda}, the right hand side is smaller than $\varepsilon$.
\end{proof}

\begin{lemma}\label{lem est Yj}
There is a $T>0$ such that for all $t \in (0,T)$ and $j \in \N$, 
\beq{eq est Yj}
\left| \frac{1}{\vol(U_j)}\int_{Y_j} \tr\left( \Phi \kappa_t(\varphi^{-1}(m), m')\lambda_t(m', m) \right)\, dm' \, dm
\right| < \varepsilon.
\eeq
\end{lemma}
\begin{proof}
Let $\delta>0$ be as in the second assumption on $U$ in Subsection \ref{sec exhaust}. 
If $(m,m') \in Y_j$, then $m' \not \in U$, so by the triangle inequality, we either have $d(m,m')\geq\delta/2$ or $d(\varphi^{-1}(m), m')\geq\delta/2$. So the left hand side of \eqref{eq est Yj} is at most equal to
\begin{multline}\label{eq est Yj 1}
 \frac{1}{\vol(U_j)}\int_{U_j} \int_{M \setminus B(m, \delta/2)} \|  \kappa_t(\varphi^{-1}(m), m')\|\| \lambda_t(m', m) \| \, dm' \, dm\\
 +
  \frac{1}{\vol(U_j)}\int_{U_j} \int_{M \setminus B(\varphi^{-1}(m), \delta/2)} \|  \kappa_t(\varphi^{-1}(m), m')\|\| \lambda_t(m', m) \| \, dm' \, dm.
\end{multline}
And for all $m \in M$,
\begin{multline*}
\int_{M \setminus B(m, \delta/2)} \|  \kappa_t(\varphi^{-1}(m), m')\|\| \lambda_t(m', m) \| \, dm' 
\\
\leq
\left(
\int_{M \setminus B(m, \delta/2)} \|  \kappa_t(\varphi^{-1}(m), m')\|^2\, dm'
\right)^{1/2}
\left(
\int_{M}\| \lambda_t(m', m) \|^2 \, dm'
\right)^{1/2}
\\
\leq
v_A(\delta/2, t)C_B t^{-a_B},
\end{multline*}
where $v_A$ is the function $v$ in Theorem \ref{thm int outside ball} for $A$, and $C_{B}$ and $a_B$ are the constants $C$ and $a$ in Definition \ref{def AL}, respectively,  for $B$. Via a similar estimate for the second term in \eqref{eq est Yj 1}, we find that \eqref{eq est Yj 1} is bounded above by
\[
v_A(\delta/2, t)C_B t^{-a_B}+v_B(\delta/2, t) C_{A} t^{-a_A}.
\]
This is independent of $j$, and goes to zero as $t \downarrow 0$. 
\end{proof}
\begin{remark}
The end of the proof of Lemma \ref{lem est Yj} is the main place where we use that the functions $\mu_l$ in Definition \ref{def asympt local}, and hence the function $v$ in Theorem \ref{thm int outside ball}, vanish to all orders in $t$ as $t \downarrow 0$. (We also use this in the proof of Lemma \ref{lem UkL alg}.)
\end{remark}

\subsection{Proof of Theorem \ref{thm asympt trace}}

\begin{lemma}\label{lem limits}
Consider maps $a\colon \N \times(0, \infty) \to \C$ and $b\colon (0, \infty) \to \C$.  For every $\varepsilon>0$, let $a_{1, \varepsilon},a_{2, \varepsilon} \colon \N \times(0, \infty) \to \C$ be such that $a = a_{1, \varepsilon}+a_{2, \varepsilon}$. Suppose that for all $\varepsilon>0$, 
\begin{enumerate}
\item there is a $T>0$ such that for all $t \in (0,T)$ and $j \in \N$, we have  $|a_{1, \varepsilon}(t,j)|<\varepsilon$;
\item for all $t \in (0,1]$, there is an $N>0$ such that for all $j \geq N$, we have $|a_{2, \varepsilon}(t,j)|<\varepsilon$;
\item for all $t>0$, 
\[\liminf_{j \to \infty} |b(t) - a(t, j)| = 0.
\]
\end{enumerate}
Then $\lim_{t\downarrow 0} b(t) = 0$.
\end{lemma}
\begin{proof}
Let $\varepsilon>0$. Let $T>0$ be as in the first point. Let $t \in (0,T) \cap(0,1]$. Let $N$ be as in the second point, for this value of $t$. By the third point, there is an $N'\geq 0$ such that for all $n \geq N'$, there is a $j \geq n$ such that $ |b(t) - a(t, j)|< \varepsilon$. Let $j \in \N$ have this property for  $n = \max(N, N')$. Then
\[
|b(t)| \leq |b(t) - a(t,j)| + |a_{1, \varepsilon}(t,j)|+|a_{2, \varepsilon}(t,j)| <3\varepsilon.
\]
\end{proof}

\begin{proof}[Proof of Theorem \ref{thm asympt trace}]
Let $A \in \cA_{-\infty}^L(S)^{\Phi}$ and  $B \in \cA_{-\infty}^L(S)$. As before, we write $\kappa_t$
for  the Schwartz kernel of $A(t)$, and  $\lambda_t$ for the Schwartz kernel of $B(t)$.
For $t>0$ and $j \in \N$, write
\begin{multline*}
a (t, j) := \\
\frac{1}{\vol(U_j)} \int_{U_j} 
\tr\left( \Phi \kappa_t(\varphi^{-1}(m), m')\lambda_t(m', m) - \Phi \lambda_t(\varphi^{-1}(m), m') \kappa_t(m', m) \right)\, dm' \, dm,
\end{multline*}
and
\[
b(t) :=  \Tr_{\Phi}^U (A(t)B(t) - B(t) A(t)). 
\]
Furhermore, if $\varepsilon>0$, then we choose $r>0$ as at the start of Subsection \ref{sec est int}, and use the sets \eqref{eq Vj} to write
\begin{multline*}
a_{1, \varepsilon} (t, j) := \\
\frac{1}{\vol(U_j)} \int_{Y_j} 
\tr\left( \Phi \kappa_t(\varphi^{-1}(m), m')\lambda_t(m', m) - \Phi \lambda_t(\varphi^{-1}(m), m') \kappa_t(m', m) \right)\, dm' \, dm,
\end{multline*}
and
\begin{multline*}
a_{2, \varepsilon} (t, j) := \\
\frac{1}{\vol(U_j)} \int_{V_j \cup W_j \cup X_j} 
\tr\left( \Phi \kappa_t(\varphi^{-1}(m), m')\lambda_t(m', m) - \Phi \lambda_t(\varphi^{-1}(m), m') \kappa_t(m', m) \right)\, dm' \, dm.
\end{multline*}

Then the first condition in Lemma \ref{lem limits} holds by Lemma \ref{lem est Yj} and the analogous statement with $\kappa_t$ and $\lambda_t$ interchanged. The second  condition in Lemma \ref{lem limits} holds by Lemmas \ref{lem est Vj}, \ref{lem est Wj} and \ref{lem est Xj},  and the versions of Lemmas \ref{lem est Wj} and \ref{lem est Xj}  with $\kappa_t$ and $\lambda_t$ interchanged. The third condition in Lemma \ref{lem limits} holds by Definition \ref{def associated}. So the claim follows by Lemma \ref{lem limits}.
\end{proof}
%
%

\subsection{The trace of the index of $D$}

We extend $\Tr^U_{\Phi}$ to matrices with entries in $\cA_{-\infty}^L(S)^{\Phi}$, by combining it with the matrix trace in the usual way.
\begin{corollary}\label{cor Tr idem}
Let $e,f \in M_{\infty}(\cA_{-\infty}^L(S)^{\Phi})$ be idempotents such that
\beq{eq diff Tr}
\lim_{t \downarrow 0} \bigl( \Tr_{\Phi}^U(e(t)) - \Tr_{\Phi}^U(f(t)) \bigr)
\eeq
converges. Let $e',f' \in M_{\infty}(\cA_{-\infty}^L(S)^{\Phi})$ be idempotents such that we have Murray--von Neumann equivalences $e \sim e'$ and $f \sim f'$. Then
\[
\lim_{t \downarrow 0} \bigl( \Tr_{\Phi}^U(e'(t)) - \Tr_{\Phi}^U(f'(t)) \bigr)
\]
converges, and equals \eqref{eq diff Tr}.
\end{corollary}
\begin{proof}
This follows directly from Theorem \ref{thm asympt trace}.
\end{proof}

Let $\widetilde{K}_0(\cA_{-\infty}^L(S)^{\Phi}) \subset K_0(\cA_{-\infty}^L(S)^{\Phi})$ be the subgroup of classes that can be represented by idempotents as in Corollary \ref{cor Tr idem}. Then by this corollary, we have a well-defined map
\[
\Tr_{\Phi}^U\colon \widetilde{K}_0(\cA_{-\infty}^L(S)^{\Phi}) \to \C
\]
given by
\[
\Tr_{\Phi}^U(x) := \lim_{t \downarrow 0} \bigl( \Tr_{\Phi}^U(e(t)) - \Tr_{\Phi}^U(f(t)) \bigr). 
\]
if $x = [e] - [f]$, with $e$ and $f$ as in Corollary \ref{cor Tr idem}.

Let $\gamma$ be the grading operator on $S$.
\begin{proposition}\label{prop index D K}
Suppose that $ \lim_{t \downarrow 0} \Tr_{\Phi}^U(\gamma e^{-tD^2})$ converges. Then
the class $\ind^L(D) \in K_0(\cA_{-\infty}^L(S)^{\Phi})$ in Definition \ref{def index K}  lies in $ \widetilde{K}_0(\cA_{-\infty}^L(S)^{\Phi})$. Furthermore, 
\beq{eq index D K}
\Tr_{\Phi}^U(\ind^L(D)) = \lim_{t \downarrow 0} \Tr_{\Phi}^U(\gamma e^{-tD^2}).
\eeq
\end{proposition}
\begin{proof}
Let 
\[
\begin{split}
e(t) &:= \begin{pmatrix}
e^{-t^2 D^- D^+} & e^{-\frac{t^2}{2} D^- D^+} \frac{1-e^{-t^2 D^- D^+}}{D^- D^+} D^- \\
e^{-\frac{t^2}{2} D^+ D^-} D^+ & 1-e^{-t^2 D^+ D^-}
\end{pmatrix}; \quad \text{and}\\
f(t) &=  \begin{pmatrix}
 0&0\\
 0&1
 \end{pmatrix}.
\end{split}
\]
Here $D^+$ is the restriction of $D$ to even-graded sections, and  $D^-$ is its restriction  to odd-graded sections. The entries of $e$  lie in the unitisation of $\cA_{-\infty}^L(S; D)$ by Proposition \ref{prop pos j A} and Lemmas \ref{lem ideal AD} and \ref{lem D Q}. 

By a standard form of the boundary map \eqref{eq bdry K} (see e.g.\ page 356 of \cite{CM90}), 
\[
\ind^L(D)  = [e] - [f].
\]
Here we use the parametrix $Q$ of $D$ from Proposition \ref{prop D Fred}.
Extending $\Tr_{\Phi}^U$ to the unitisation of $\cA_{-\infty}^L(S; D)$ in the usual way, by defining it to map the unit element to $0$, we see that \eqref{eq index D K} holds.
\end{proof}

\begin{definition}\label{def index}
If $\ind^L(D) \in \widetilde{K}_0(\cA_{-\infty}^L(S)^{\Phi})$, then
the \emph{localised $\Phi$-index} of $D$ is the number
\[
\ind_{\Phi}^U(D) := 
\Tr_{\Phi}^U(\ind^L(D) ).
\]
\end{definition}


Proposition \ref{prop inv index zero} has an important consequence.
\begin{corollary}
If $D$ is invertible, then $\ind_{\Phi}^U(D) =0$.
\end{corollary}

\section{An index theorem}\label{sec index}

\subsection{The main result}

We will use some standard heat kernel asymptotics. We denote the Atiyah--Segal--Singer integrand associated to $D$ and $\Phi$ by $\AS_{\Phi}(D)$. See e.g.\ Theorem 6.16 in \cite{BGV} for a general expression. 
\begin{proposition}\label{prop asymp exp}
Let $\kappa_t$ be the Schwartz kernel of $e^{-tD^2}$. There is a function $R\colon (0, \infty) \to (0, \infty)$ such that $\lim_{t \downarrow 0}R(t) = 0$, and  for all
relatively compact open subsets $V \subset M$, the following hold.
\begin{enumerate}
\item[(a)] If there is a $\delta>0$ such that $d(\varphi(m), m)\geq \delta$ for all $m \in M$, then
\[
\left| 
\int_V \tr(\gamma \Phi \kappa_t(\varphi^{-1}(m), m))\, dm \right| \leq  \vol(V) R(t).
\]
%
%
\item[(b)] If $\varphi$ preserves an orientation on $M$, and is contained in a compact group of isometries of $M$, and  $\Phi$ is contained in a compact group of isometries of $S$, then
\beq{eq asymp exp}
\left| 
\int_V \tr(\gamma \Phi \kappa_t(\varphi^{-1}(m), m))\, dm - \int_{V^{\varphi} } \AS_{\Phi}(D) \right| \leq  \vol(V) R(t).
%
\eeq
%
\end{enumerate}
\end{proposition}
\begin{proof}
Part (a) follows from the global Gaussian upper bound for heat kernels in manifolds of bounded geometry that was also used in the proof of Proposition \ref{prop pos j A}.

Now suppose that $\varphi$ and $\Phi$ lie in compact groups, and that $\varphi$ preserves an orientation. 
Then it follows from standard heat kernel asymptotics that for every relatively compact open subset $V \subset M$, there is  a function $R$ with the property in part (b), possibly depending on $V$.  See Theorems 6.11 and 6.16 in \cite{BGV}, where one uses that the asymptotic expansion of the heat kernel holds with respect to supremum norms on compact sets.
The fact that $R$ may be chosen independent of $V$ follows from bounded geometry. This implies that the  heat kernel $\kappa_t$ and its derivatives are uniformly bounded, which implies the desired independence of $V$ as in Section 2 of \cite{Roe88I}. Compare also Proposition 4.2 in \cite{CGRS14}. 
\end{proof}

In the non-equivariant setting of \cite{Roe88I}, the relevant trace of $\gamma e^{-tD^2}$ is independent of $t$ by Proposition 8.1 in \cite{Roe88I}. In the current situation, we need to take the limit $t \downarrow 0$ into account explicitly. This leads to an additional assumption.
\begin{proposition}\label{prop index limit}
\begin{enumerate}
\item[(a)] If there is a $\delta>0$ such that $d(\varphi(m), m)\geq \delta$ for all $m \in M$, then 
\[
\lim_{t\downarrow 0} \Tr_{\Phi}^U(\gamma e^{-tD^2})  = 0.
\]
\item[(b)] Suppose that $\varphi$ preserves an orientation on $M$, and that $\varphi$ and $\Phi$ lie in compact groups of isometries. 
Suppose that
\beq{eq limit AS}
\lim_{j \to \infty} \frac{1}{\vol(U_j)} \int_{U_j^{\varphi}} \AS_{\Phi}(D)
\eeq
converges. Then 
\[
\lim_{t\downarrow 0} \Tr_{\Phi}^U(\gamma e^{-tD^2}) 
\]
 converges to the same value as \eqref{eq limit AS}. 
 \end{enumerate}
%
\end{proposition}
\begin{proof}
Let $R$ be as in Proposition \ref{prop asymp exp}.
Let $\varepsilon>0$, and let $T>0$ be such that $R(t)< \varepsilon$ for all $t \in (0, T)$. Fix $t \in (0,T)$.

In case (a), choose $j \in \N$ such that
\[
\left|  \Tr_{\Phi}^U(\gamma e^{-tD^2}) -  \frac{1}{\vol(U_j)}\int_{U_j} \tr(\gamma \Phi \kappa_t(\varphi^{-1}(m), m))\, dm  \right|<\varepsilon.
\]
Then by Proposition \ref{prop asymp exp}(a), with $V = U_j$, 
\[
|\Tr_{\Phi}^U(\gamma e^{-tD^2}) | < 2 \varepsilon.
\]

In case (b), Proposition \ref{prop asymp exp}(b) implies that
for all $j \in \N$, 
\beq{eq limit index 1}
\left| \frac{1}{\vol(U_j)}\int_{U_j} \tr(\gamma \Phi \kappa_t(\varphi^{-1}(m), m))\, dm -  \frac{1}{\vol(U_j)}  \int_{U_j^{\varphi} }  \AS_{\Phi}(D) \right|  < \varepsilon.
\eeq
%
%
Let $a$ be the value of the limit \eqref{eq limit AS}. 
Let $N \in \N$ be such that for all $j \geq N$, 
\beq{eq limit index 2}
\left|  \frac{1}{\vol(U_j)} \int_{U_j^{\varphi}} \AS_{\Phi}(D) - a\right|< \varepsilon.
\eeq
 By definition of $ \Tr_{\Phi}^U$, there is a $j\geq N$ such that 
\beq{eq limit index 3}
\left|  \Tr_{\Phi}^U(\gamma e^{-tD^2}) -  \frac{1}{\vol(U_j)}\int_{U_j} \tr(\gamma \Phi \kappa_t(\varphi^{-1}(m), m))\, dm  \right|<\varepsilon.
\eeq
The inequalities \eqref{eq limit index 1}--\eqref{eq limit index 3} for such a $j$ imply that
\[
\left|
\Tr_{\Phi}^U(\gamma e^{-tD^2}) -  
a
\right|
<3\varepsilon.
\]
\end{proof}

Our main result is the following index theorem, which follows from Propositions \ref{prop index D K} and \ref{prop index limit}.
\begin{theorem}\label{thm index}
\begin{enumerate}
\item[(a)] If there is a $\delta>0$ such that $d(\varphi(m), m)\geq \delta$ for all $m \in M$, then $\ind^L(D) \in \widetilde{K}_0(\cA_{-\infty}^L(S)^{\Phi})$, and
\[
\ind_{\Phi}^U(D)  = 0.
\]
\item[(b)] Suppose that  $\varphi$ preserves an orientation on $M$, and that $\varphi$ and $\Phi$ lie in compact groups of isometries. 
Suppose that
\beq{eq limit top index}
\lim_{j \to \infty} \frac{1}{\vol(U_j)} \int_{U_j^{\varphi}} \AS_{\Phi}(D) 
\eeq
converges. 
Then $\ind^L(D) \in \widetilde{K}_0(\cA_{-\infty}^L(S)^{\Phi})$, and
\beq{eq index}
\ind_{\Phi}^U(D) = \lim_{j \to \infty} \frac{1}{\vol(U_j)} \int_{U_j^{\varphi}} \AS_{\Phi}(D).
\eeq
\end{enumerate}
\end{theorem}

\begin{remark}
Convergence of the limit \eqref{eq limit top index} depends on the choice of $U$, so it is possible that $U$ can be adapted so that this condition holds. 
\end{remark}

\begin{remark}\label{rem U=M}

In the case where $U = M$, 
\[
 \frac{1}{\vol(U_j)} \int_{U_j^{\varphi}} \AS_{\Phi}(D) =
\frac{1}{\vol(M_j)} \int_{M_j^{\varphi}} \AS_{\Phi}(D).  
\]
Because of bounded geometry, the integrand is bounded. Hence the absolute value of the  right hand side is less than or equal to a constant times
\[
\frac{\vol(M_j^{\varphi})}{\vol(M_j)}. 
\]
If  $\varphi$ is not the identity map, then 
 goes to zero as $j \to \infty$ in many cases, because $M^{\varphi}$ has lower dimension than $M$. In such cases, Theorem \ref{thm index} implies that $\ind_{\Phi}^M(D) = 0$. For this reason, it is important to allow $U$ to be different from $M$. Taking $U$ to be a tubular neighbourhood of $M^{\varphi}$ seems to be a natural choice.
\end{remark}

\subsection{Consequences and applications}

Assume from now on that  $\varphi$ preserves an orientation on $M$, and that $\varphi$ and $\Phi$ lie in compact groups of isometries.

Our main interest is in the case where the fixed-point set $M^{\varphi}$ is noncompact, but in the case where it is compact, we obtain a natural consequence of Theorem \ref{thm index}.
\begin{corollary}\label{cor index cpt}
Suppose that $M^{\varphi}$ and $\bar U$ are compact. 
Then $\ind^L(D) \in \widetilde{K}_0(\cA_{-\infty}^L(S))$, and
\[
\ind_{\Phi}^U(D) = \frac{1}{\vol(U)} \int_{M^{\varphi}} \AS_{\Phi}(D).
\]
\end{corollary}
\begin{proof}
If $M^{\varphi}$ is compact, then $U_j = U$ for large enough $j$. So the claim follows from Theorem \ref{thm index}.
\end{proof}

Corollary \ref{cor index cpt} implies that the index defined in this paper generalises the one constructed in \cite{HW}. 
\begin{corollary}
Suppose that 
 $M^{\varphi}$ and $\bar U$ are compact. Then $\ind_{\Phi}^U(D)$ equals the index of Definition 2.7 in \cite{HW} divided by $\vol(U)$.
\end{corollary}
\begin{proof}
In this setting, Corollary \ref{cor index cpt} in this paper and Theorem 2.16 in \cite{HW} imply that the two indices equal the same topological expression, up to a factor $1/\vol(U)$.
\end{proof}

The number $\ind_{\Phi}^U(D)$ also generalises the index from \cite{Roe88I}.
\begin{corollary}
Suppose that $\varphi = \Id_M$ and $\Phi = \Id_S$ are the identity maps on $M$ and $S$, respectively. Then $\ind^L(D) \in \widetilde K_0(\cA_{-\infty}^L(D)^{\Phi})$, and
\[
\ind_{\Id_S}^M(D) = \dim_{\tau}(\Ind(D)),
\]
where the right hand side is as in Section 7 and 8 of \cite{Roe88I}.
\end{corollary}
\begin{proof}
Now $\Tr_{\Id_S}^M$ is the trace $\tau$ in Theorem 6.7 in \cite{Roe88I}. By Proposition 8.1 in \cite{Roe88I}, 
\[
\Tr_{\Id_S}^M(\Ind(D)) = \Tr_{\Id_S}^M(\gamma e^{-tD^2})
\]
for all $t>0$. In particular, $\lim_{t \downarrow 0}  \Tr_{\Id_S}^M(\gamma e^{-tD^2})$ converges. So the claim follows from Proposition \ref{prop index D K}.
\end{proof}

\begin{corollary}\label{cor inf vol}
If $U$ has infinite volume and
\[
\int_{M^{\varphi}} \AS_{\Phi}(D)
\]
converges, then $\ind_{\Phi}^U(D) = 0$.
\end{corollary}
\begin{proof}
In this case, 
\[
\lim_{j \to \infty} \frac{1}{\vol(U_j)} \int_{U_j^{\varphi}} \AS_{\Phi}(D) = 0.
\]
\end{proof}


Finally, we obtain an obstruction to positive scalar curvature metrics. Note that the statement does not involve the algebras, asymptotic trace and indices defined in this paper; they are only used in its proof.
\begin{corollary}
If $M$ is $\Spin$ and has a complete Riemannian metric with uniformly positive scalar curvature, then for every isometry $\varphi \colon M \to M$ that has a lift $\Phi$ to the spinor bundle and lies in a compact group of isometries, the limit
\[
\lim_{j \to \infty} \frac{1}{\vol(U_j)} \int_{U_j^{\varphi}} \frac{\hat A(M^{\varphi})}{\det(1-\Phi e^{-R})^{1/2}}
\]
equals zero if it converges. Here $R$ is the curvature of the connection on the normal bundle of $M^{\varphi}$ induced by the Levi--Civita connection.
\end{corollary}
\begin{proof}
If $D$ is the $\Spin$-Dirac operator, then
\[
\AS_{\Phi}(D) =  \frac{\hat A(M^{\varphi})}{\det(1-\Phi e^{-R})^{1/2}}.
\]
So the claim follows from the Lichnerowicz formula, Proposition \ref{prop inv index zero} and Theorem \ref{thm index}.
\end{proof}

\bibliographystyle{plain}

\bibliography{mybib}

\end{document}